\documentclass[11pt,preprint]{amsart}
\usepackage{natbib}
\usepackage{graphicx}
\usepackage{xargs}
\usepackage[active]{srcltx}

\usepackage[shortlabels]{enumitem} 

\usepackage[utf8]{inputenc}
\usepackage[english]{babel}
\usepackage{latexsym}
\usepackage{amssymb}
\usepackage{upgreek}
\usepackage{bbm}
\usepackage{mathrsfs}
\usepackage[norelsize,ruled,vlined]{algorithm2e}
\usepackage{amsmath}
\usepackage[textwidth=4cm,textsize=footnotesize]{todonotes}
\usepackage{accents}
\usepackage{dsfont}
\usepackage{hyperref}

\usepackage{aliascnt}
\usepackage{cleveref}
\makeatletter
\newtheorem{theorem}{Theorem}
\crefname{theorem}{theorem}{Theorems}
\Crefname{Theorem}{Theorem}{Theorems}

\newaliascnt{lemma}{theorem}
\newtheorem{lemma}[lemma]{Lemma}
\aliascntresetthe{lemma}
\crefname{lemma}{lemma}{lemmas}
\Crefname{Lemma}{Lemma}{Lemmas}

\newaliascnt{corollary}{theorem}
\newtheorem{corollary}[corollary]{Corollary}
\aliascntresetthe{corollary}
\crefname{corollary}{corollary}{corollaries}
\Crefname{Corollary}{Corollary}{Corollaries}

\newaliascnt{proposition}{theorem}

\aliascntresetthe{proposition}
\crefname{proposition}{proposition}{propositions}
\Crefname{Proposition}{Proposition}{Propositions}

\newaliascnt{definition}{theorem}

\aliascntresetthe{definition}
\crefname{definition}{definition}{definitions}
\Crefname{Definition}{Definition}{Definitions}

\newaliascnt{definitionProposition}{theorem}

\aliascntresetthe{definitionProposition}
\crefname{Proposition and Definition}{Proposition and Definition}{Proposition and Definition}
\Crefname{Proposition and Definition}{Proposition and Definition}{Proposition and Definition}

\newaliascnt{remark}{theorem}

\aliascntresetthe{remark}
\crefname{remark}{remark}{remarks}
\Crefname{Remark}{Remark}{Remarks}

\crefname{example}{example}{examples}
\Crefname{Example}{Example}{Examples}

\crefname{figure}{figure}{figures}
\Crefname{Figure}{Figure}{Figures}

\newtheorem{assumption}{\textbf{H}\hspace{-3pt}}
\Crefname{assumption}{\textbf{H}\hspace{-3pt}}{\textbf{H}\hspace{-3pt}}
\crefname{assumption}{\textbf{H}}{\textbf{H}}

\Crefname{assumptionL}{\textbf{L}\hspace{-3pt}}{\textbf{L}\hspace{-3pt}}
\crefname{assumptionL}{\textbf{L}}{\textbf{L}}

\Crefname{assumptionG}{\textbf{G}\hspace{-3pt}}{\textbf{G}\hspace{-3pt}}
\crefname{assumptionG}{\textbf{G}}{\textbf{G}}

\newcommand{\ocint}[1]{\left(#1\right]}

\newcommand{\ccint}[1]{\left[#1\right]}

\def\wrt{with respect to}
\def\rset{\mathbb{R}}
\def\zset{\mathbb{Z}}

\newcommand\parenth[1]{\left(#1\right)}
\newcommand\po{\left(}
\newcommand\pf{\right)}
\newcommand\crochet[1]{\left[#1\right]}
\newcommand\accol[1]{\left\{#1\right\}}

\newcommand\intoo[2]{\left(#1,#2\right)}
\newcommand\intof[2]{\left(#1,#2\right]}

\newcommand\intff[2]{\left[#1,#2\right]}

\newcommand\indic[1]{\mathds{1}_{#1}}
\newcommand\indiacc[1]{\mathds{1}_{\{#1\}}}

\newcommand\sumforce[2]{\underset{#1}{\overset{#2}{\sum}}}

\newcommand\argmin[2]{\underset{#1}{\text{argmin}}\left\{#2\right\}}

\def\eqsp{\,}

\newcommand\abs[1]{\left\lvert#1\right\rvert}
\newcommand\norm[2]{\left\|{#2}\right\|_{#1}}

\newcommand\pent[1]{\left\lfloor{#1}\right\rfloor}

\newcommand\rmd{\mathrm{d}}
\newcommand\rme{\mathrm{e}}

\newcommand\set[2]{\left\lbrace #1 \, : \, #2 \right\rbrace}
\newcommand\reel{\mathbb{R}}
\def\rset{\mathbb{R}}

\newcommand\reelp{\mathbb{R_{+}}}

\newcommand\relatif{\mathbb{Z}}

\newcommand\iid{i.i.d. }

\newcommand\prob[1]{\mathbb{P}\left(#1\right)}

\newcommand\esp[1]{\mathbb{E}\left[#1\right]}

\newcommand\var[1]{\mathbb{V}\mathrm{ar}\left(#1\right)}

\def\targetmoment{m}
\def\momentestim{\widehat{m}}
\def\tdens{f}
\def\param{M}

\def\adens{\tdens^\param}
\def\edens{\widehat{\tdens}}

\def\lepskiparam{\widehat{\param}}

\def\etat{k}
\newcommand\betaker[2]{\mathbb{B}^{#1,#2}}

\def\cstholder{C_{\beta,L}}

\newcommand{\Pbf}{\mathbf{P}}
\newcommand{\Ebf}{\mathbf{E}}
\newcommand{\Pbb}{\mathbb{P}}
\newcommand{\Ebb}{\mathbb{E}}
\newcommand{\Lcal}{\mathcal{L}}

\def\ie{i.e.}

\usepackage{amsaddr}

\begin{document}
\title{Density estimation for RWRE}
\author{A. Havet }
\address{Centre de Math\'ematiques Appliqu\'ees, \'Ecole polytechnique, CNRS Universit\'e Paris-Saclay}
\email{antoine.havet@polytechnique.edu}
\author{ M. Lerasle}\address{Laboratoire de Math\'ematiques d'Orsay, Univ. Paris-Sud, CNRS, Universit\'e Paris Saclay}
\email{matthieu.lerasle@polytechnique.edu}
\author{\'Eric Moulines}\address{Centre de Math\'ematiques Appliqu\'ees, \'Ecole polytechnique, CNRS Universit\'e Paris-Saclay}
\email{eric.moulines@polytechnique.edu}
\maketitle
\begin{abstract}
 We consider the problem of non-parametric density estimation of a random environment from the observation of a single trajectory of a random walk in this environment. We first construct a density estimator using the beta-moments.  We then show that the Goldenshluger-Lepski method can be used to select the beta-moment. We prove non-asymptotic bounds for the supremum norm of these estimators for both the recurrent and the transient to the right case.  A simulation study supports our theoretical findings.
\end{abstract}
\section{Introduction}
%

A random walk in random environment (RWRE) provides a simple model to describe various problems, for example the propagation of heat, diffusion of  matter through a physical medium, or DNA-unzipping experiments. In many situations of practical interest,  the medium in which the system evolve is very irregular. In these cases, it is  natural  to model these irregularities as a random environment. The definition of an RWRE involves two ingredients: (i) the environment, which is an \iid\ sample of some unknown disribution $\nu$ and (ii) the random-walk whose transition probabilities are dermined by the environment.  This paper considers the problem of recovering the density $\tdens$ of the distribution $\nu$ of the environment of a RWRE on $\relatif$ based on the observation of a single trajectory of the RWRE.

Since their introduction by Chernov \cite{Chernov:1967}, RWRE have been widely studied in the probability literature; see for example \cite{Zei:2012} for a recent overview of probabilistic results on RWRE in $\relatif$ and more generally in $\relatif^d$. On the other hand, statistical inference for RWRE has emerged only recently with the appearance of RWRE in several statistical models such as DNA-unzipping experiments or DNA- polymerase phenomenon \cite{Alemany2012, Baldazzi2007, Baldazzi2006, Huguet2009, Koch2002}. In these applications, the main task is usually to recover the environment itself and, when it is the case, an estimator of $\nu$ can be considered as a preliminary  step in the construction of a an empirical Bayes estimator.

More precisely, the problem of recovering $\nu$ was originally considered in \cite{Ade_Enr:2004} who introduced an estimator of the moments of the distribution. The article \cite{Ade_Enr:2004} considered random walks in general state spaces but the estimators they deduced had poor statistical performance. More recently, \cite{Fal_Lou_Mat:2014, Fal_Glo_Lou:2014, comets:falconnet:2014,Comets2016Maximum} considered the random walk on $\relatif$ and studied the parametric estimation of $\nu$. More precisely, they proved consistency and asymptotic normality of the maximum likelihood estimator in the ballistic and sub-ballistic regimes as well as its efficiency in the ballistic regime (see Section~\ref{sec:setting} for details regarding these notions). Some of these consistency results have been extended to i.i.d. environments in recurrent regimes in \cite{Comets2016Maximum} for a slightly different estimator and also to Markovian environments in \cite{AndLouMat2015}.


It is shown in \cite{comets:falconnet:2014} and \cite{diel:lerasle:2016}  that the beta-moments of the distribution $\nu$ can be estimated consistently from a single trajectory of the RWRE. In \cite{diel:lerasle:2016}, these estimators of moments are used to construct an estimator of the cumulative distribution function (c.d.f.) of the random environment. In this paper, we use these moment estimators to construct estimators of the probability density.

Recovering the probability density  of a distribution from its moment is a classical problem in statistics called \textit{moment reconstruction} \cite{akhiezer:1965}. Density estimators of a distribution based on its beta-moments (also referred to as beta-kernel estimators) have already been studied in the classical i.i.d. direct  observation setting, where a sample $(Z_1,\dots,Z_n)$ of the unknown distribution is available for the inference (see for example \cite{chen1999BetaKernel,bouezmarni:rolin:2003,mnatsakanov:2008} and the references therein). In the direct observation setting, \cite{chen:1999} proposed to estimate the density  $f$ at a point $x \in \ccint{0,1}$  by a single appropriately chosen $\beta$-moment
\[
\hat{f}_h(x)= \frac{1}{n} \sum_{i=1}^n \betaker{\frac{x}{h}+1}{\frac{1-x}{h}+1}(Z_i) \eqsp,
\]
where $h$ is the smoothing bandwidth and for any $(a,b)$ in $\reel_+^*\times\reel_+^*$,  $\betaker{a}{b}$ is the beta function
\begin{equation}
\label{def:BetaKer}
\betaker{a}{b}(u) = \frac{\Gamma(a+b)}{\Gamma(a)\Gamma(b)} u^{a-1} (1-u)^{b-1} \indic{0\leqslant u \leqslant 1} \eqsp.
\end{equation}
\cite{chen1999BetaKernel} introduced this estimator for a density function on $\ccint{0,1}$ in order to remove the boundary bias of the standard kernel estimator. He proved that this estimator is free of boundary bias.  This beta-kernel density estimator has later been shown to be minimax under classical regularity assumptions on the density $f$ by \cite{bertin:klutchnikoff:2010}.  \cite{bertin:klutchnikoff:2010}.

Our contribution is threefold. First, we propose a non-parametric density estimator for the random environment of a RWRE. To our best knowledge, this is the first density estimator of the distribution of the environment. Our proposed methodolgy follows the same steps than the classical density estimation problem but our construction is not a "kernel" estimator, contrary to the "direct-observation" case.
The unknown density $f$ of $\nu$ is first approximated at every point $x\in(0,1)$ by a sequence of properly chosen beta kernels following  \cite{chen1999BetaKernel,mnatsakanov:2008}. The $\beta$-moments are estimated using the methods outlined in \cite{diel:lerasle:2016}. The proposed density estimator is a function of  bandwidth parameter. We propose and analyze a method to choose automatically this bandwidth parameter, based on the Goldenshluger-Lepski algorithm \cite{goldenshluger:lepski:2008}. This extra step of selection is crucial here to optimize our bound without any prior knowledge of the regime of the walk or the regularity of the unknown density (see Section~\ref{sec:MainRes} for details). The Goldenshluger-Lepski method was also used in \cite{bertin:klutchnikoff:2014} to build adaptive estimators in the i.i.d. setting based on the minimax estimators of \cite{bertin:klutchnikoff:2010}.

Our second contribution is a derivation of the first non-parametric risk bounds for any density estimator of $f$ based on the observation of a single trajectory of a RWRE. These are exploiting the results recently developed in \cite{diel:lerasle:2016} for the stochastic part of the risk. We are of course also using some of the results on the beta-kernel estimator presented in \cite{Bouezmarni2003Consistency,mnatsakanov:2008,bertin:klutchnikoff:2010}.
Our rates do not match minimax rates of density estimation in i.i.d. setting.
Nevertheless, our results outperform those obtained in \cite{Comets2016Maximum} in the recurrent regime where only consistency results for very particular densities were derived. Furthermore, our estimator does not require a prior knowledge on the regime of the walk contrary to that of \cite{Comets2016Maximum} - the results in this work presuppose that RWRE is recurrent, which is not required here.

We finally investigate the numerical behaviour of our estimator in a Monte Carlo experiment.
It is interesting to notice that our estimator behaves reasonably 
in ballistic regimes even if the chain has a linear drift. This emphasizes the difference between population and individual parameters estimation problems. It is clear that recovering the environment itself is impossible in the ballistic regime, yet estimating the law of the random environment is still possible.
Related results  have already been reported in a parametric setting for example in \cite{Fal_Lou_Mat:2014} but it comes perhaps as a surprise that the same behaviour is observed in the  non-parametric setting.
Both theoretical results and this first simulations consider the case where the chain is observed until it reaches a fixed site $n$. Depending on the regime, the number of observed steps of the random walk can therefore be very different, it is typically $O(n)$ in the ballistic regime and $\rme^{O(\sqrt{n})}$ in the recurrent regime. Our simulations also illustrate how the quality of estimation deteriorates when going from nearly recurrent regime to nearly ballistic regime.

The paper is organised as follows. In Section~\ref{sec:setting}, we summarize some basic results on RWRE on $\relatif$ and on the likelihood of the observations before defining our basic estimator.
In~\Cref{sec:estimator-construction}, we introduce our estimator. In Section~\ref{sec:MainRes}, we state our main results : convergence rates for the basic estimators and for the Goldenshluger-Lepski estimator derived from these. Section~\ref{sec:SimStud} presents Monte Carlo experiment supporting our theoretical claims. Proofs of the main results are gathered in Section~\ref{sec:Proof}.

\section{Random walks in random environment (RWRE)}\label{sec:setting}
Denote by  $E = \intoo{0}{1}^{\mathbb{Z}}$  the \textit{set of environments}  endowed with $\mathcal{E}=\mathcal{B}\left(\intoo{0}{1}\right)^\mathbb{Z}$, the $\sigma$-field generated by the cylinders. 
Denote by $\left(X_t\right)_{t\in\relatif_+}$ the canonical process of the space $\mathbb{Z}^{\mathbb{Z}_+}$ endowed with the $\sigma$-field generated by the cylinders. For $\omega \in E$, we define the random walk in the environment $\omega$ as the time-homogeneous Markov chain with transition probabilities given, for all $x,y \in \zset$ by
\[
p_\omega(x,y)= \begin{cases} 
\omega_x & \text{if $y=x+1$} \\
1-\omega_x & \text{if $y=x-1$} \\
0 & \text{otherwise} 
\end{cases}
\]
In all the sequel, it is assumed that the Markov chain is started from $0$.
That is to say, given the environment $\omega$, the random walk currently at point $x \in \zset$, will make a one-unit step to the right with probability $\omega_x$ or to the left with probability $1-\omega_x$.
The probability measure $\Pbf_\omega$ that determines the distribution of the random walk in a given
environment $\omega$ is called the \emph{quenched distribution}. 

Here the environment is determined by the sequence of random variables $\omega= (\omega_y)_{y \in \zset}$.
In this paper, we assume that the random probabilities $(\omega_x)_{x \in \zset}$ are \iid\ with common distribution $\nu$. The \textit{environment} $\omega$ is thus distributed according to $\Pbb^\nu=\nu^{\otimes\relatif}$.\

By averaging the quenched probability with respect to the environment distribution we obtain the \emph{annealed distribution} $\Pbf^\nu$ on $\left(2^\mathbb{Z}\right)^{\otimes\relatif_+}$ defined by:
\[
\Pbf^\nu(\cdot) = \int_E \Pbf_\omega(\cdot) \Pbb^\nu(\mathrm{d}\omega).
\]
Expectation with respect to the annealed probability measure $\Pbf^\nu$ will be denoted $\Ebf^\nu$.
Note that the random walk $\left(X_t\right)_{t\in\relatif_+}$ is a Markov chain only conditionally on the fixed environment (\ie\ \wrt\ to the quenched distribution $\Pbf_\omega$) but the Markov property fails under the annealed probability measure $\Pbf^\nu$. This is because the past history canot be neglected, as it provides information on the environment.  
The random walk in \iid random environment on $\mathbb{Z}$ is the random sequence $\left(X_t\right)_{t\in\relatif_+}$ considered under annealed distribution $\Pbf^\nu$.

The asymptotic behaviour of the walk $(X_t)_{t\in \relatif_+}$ depends on the distribution of the ratio
\begin{equation}
\label{eq:definition-rho-0}
\rho_0=\frac{1-\omega_0}{\omega_0} \eqsp.
\end{equation}
More precisely, if $\Ebb^\nu\crochet{|\log\rho_0|}$ is finite, Solomon \cite{solomon:1975} proved the following classification:
\begin{enumerate}[(i)]
	\item if $\Ebb^\nu\crochet{\log\rho_0} \ne 0$, then $(X_t)_{t\in \relatif_+}$ is transient ($\Pbf^\nu$-a.s.); moreover if $\Ebb^\nu\crochet{\log\rho_0} < 0$, then $\lim_{t\to+\infty}X_t=+\infty$, $\Pbf^\nu$-a.s., the process $(X_t)_{t \in \zset_+}$ is transient to the right.
	\item if $\Ebb^\nu\crochet{\log\rho_0}=0$,  then $(X_t)_{t\in \relatif_+}$ is recurrent; moreover, $\limsup_{t\to+\infty}X_t=+\infty$ and $\liminf_{t\to+\infty}X_t=-\infty$, $\Pbf^\nu$-a.s.
\end{enumerate}
In the transient case, the random walk escapes to infinity, and it is reasonable to ask at what speed. For the simple RWRE, the asymptotic velocity was obtained by \cite{solomon:1975,kesten:kozlov:spitzer:1975}.

The transient case may further be divided into two sub-cases, called ballistic and sub-ballistic which correspond to a linear and sub-linear speed for the random walk.
Denote by $T_n$ the first hitting time of $n\in \relatif_+$,
\begin{equation}
\label{eq:definition-T-n}
T_n=\inf\accol{t\in\relatif_+,\ X_t=n} \eqsp.
\end{equation}
Assuming that that the RWRE is transient $\Ebb^\nu[\log \rho_0] < 0$, we can distinguish 
\begin{enumerate}
\item If $\Ebb^\nu[\rho_0] < 1$, then $\Pbf^\nu$-a.s., 
\[
T_n/n \to \frac{1 + \Ebb^\nu[\rho_0]}{1 - \Ebb^\nu[\rho_0]}
\]
and the RWRE is called \emph{ballistic}.
\item If $\Ebb^\nu[\rho_0] \geq 1$, then $\Pbf^\nu$-a.s., $T_n/n \to \infty$ and the RWRE is called sub-ballistic.
\end{enumerate}
The fluctuations of $T_n$ may be characterized more precisely.  Suppose that the distribution of $\log \rho_0$ is non arithmetic (that is the group generated by the support of $\log\rho_0$ is dense in $\rset$) and that there exists $\kappa\in(0,\infty)$ such that
\begin{align}\label{eq:kappa}
\Ebb^\nu\crochet{\rho_0^\kappa}=1\  \text{ and }\ \Ebb^\nu\crochet{\rho_0^\kappa\log^+(\rho_0)}<\infty
\end{align}
where $\log^+(x)=\log(x\vee 1)$. A simple convexity argument shows that if  $\kappa$ exists then it is unique. The value of $\kappa$ determines the asymptotic behaviour of $(X_t)_{t\in \relatif_+}$.
Then, it follows from \cite{kesten:kozlov:spitzer:1975} that
\begin{enumerate}
	\item if $\kappa<1$,
	$T_n/n^{1/\kappa}$ and $X_t/t^\kappa$ converge in $\Pbf^\nu$-distribution to some non trivial distribution,
	\item if $\kappa=1$,
	$T_n/n\log n$ and $(\log t/t)X_t$ converge in $\Pbf^\nu$-probability to a non zero constant,
	\item if $\kappa>1$,
	$T_n/n$ and $X_t/t$ converge in $\Pbf^\nu$-probability to a non zero constant.
\end{enumerate}
In the first two cases, the random walk is sub-ballistic and in the last case, $(X_t)_{t\in \relatif_+}$ is  ballistic as $T_n$ and $X_t$ grow linearly.

When the RWRE is recurrent, the fluctuations of $(X_t)_{t\in \relatif_+}$ have been evaluated by Sinai \cite{Sinai:1982}: suppose that $\Ebb^\nu\crochet{\log\rho_0}=0$, $\Ebb^\nu\crochet{\log^2\rho_0}>0$ and that the support of $\rho_0$ is included in $(0,1)$, then $X_t/(\log t)^2$ converges in $\Pbf^\nu$-distribution to a non trivial limit (see also \cite{Zei:2012} for some extensions under relaxed versions of this assumption).

Our results hold under the assumptions of \cite{diel:lerasle:2016}, that is those of \cite{kesten:kozlov:spitzer:1975}, in the transient case, and a slightly weaker version of those of \cite{Sinai:1982} in the recurrent case. Consider the following assumptions:
\begin{assumption}
\label{assum:recurrent}
$\Ebb^\nu\crochet{\log\rho_0}=0$,  $\Ebb^\nu\crochet{\log^2\rho_0}>0$ and there exists $a>0$, such that $\Ebb^\nu\crochet{\rho_0^a}+\Ebb^\nu\crochet{\rho_0^{-a}}<\infty$.
\end{assumption}
\begin{assumption}
\label{assum:transient-to-right}
$\Ebb^\nu\crochet{\log\rho_0}<0$,  the distribution of $\log \rho_0$ is non arithmetic
and there exists $\kappa\in(0,\infty)$ such that $\Ebb^\nu\crochet{\rho_0^\kappa}=1$ and  $\Ebb^\nu\crochet{\rho_0^\kappa\log^+(\rho_0)}<\infty$.
\end{assumption}
Under \Cref{assum:recurrent}, $(X_t)_{t \in \relatif_+}$ is recurrent.
Under \Cref{assum:transient-to-right}, $(X_t)_{t\in \relatif_+}$ is transient to the right.
In both cases, the hitting time $T_n$ is almost surely finite for any $n\in\relatif_+$.
We assume that the RWRE is observed until the first time it hits the state $n$. The observations are thus given by $X_0,\ldots,X_{T_n}$.

Following \cite{comets:falconnet:2014}, for any $t_0>0$ and $y\in \relatif$, define by
$L(t_0,y)$ and $R(t_0,y)$ the number of steps to the left and to the right until time $t_0$ and from site $y$:
\begin{align*}
	L(t_0,y)&=\sum_{0\leqslant t\leqslant t_0-1}\indiacc{X_t=y,X_{t+1}=y-1},\\
	R(t_0,y)&=\sum_{0\leqslant t\leqslant t_0-1}\indiacc{X_t=y,X_{t+1}=y+1} \eqsp.
\end{align*}
The likelihood $\Lcal_{\nu}\parenth{X_{0},\ldots,X_{T_n}}$ of the observations (where $T_n$ is defined in \eqref{eq:definition-T-n}) can be expressed as follows \cite{comets:falconnet:2014}
\begin{align*}
\int\parenth{\prod_{y\in\relatif}\omega_y^{R(T_n,y)}(1-\omega_y)^{L(T_n,y)}}\Pbb^\nu(\rmd\omega)
=\prod_{y\in\relatif}\int_0^1a^{R(T_n,y)}(1-a)^{L(T_n,y)}\nu(\rmd a)\enspace.
\end{align*}
Now, since $T_n$ is the hitting time of the state $n$, for any $y \geq n$, $L(T_n,y)=0$, and
\[L(T_n,y+1)=
\begin{cases}
 R(T_n,y),\ &\text{for all $y<0$}\\
 R(T_n,y)-1,\  &\text{for all $y\in [0,n-1]$}
\end{cases}
\enspace.\]
Hence the likelihood may be rewritten as
\[\Lcal_{\nu}\parenth{X_{0},\ldots,X_{T_n}}=\prod_{y\leqslant n-1}\int_0^1a^{L(T_n,y+1)+\indiacc{y\geqslant 0}}(1-a)^{L(T_n,y)}\nu(\rmd a)\enspace.\]
The collection $(L(T_n,y))_{y\leqslant n}$ is therefore a sufficient statistic that will be the building block of our estimation method. An important result of \cite{kesten:kozlov:spitzer:1975}, central in the analysis of the parametric case  derived in \cite{comets:falconnet:2014} is that, under the annealed distribution $\Pbf^\nu$, the process
$$(Z_y^n)_{0\leqslant y\leqslant n}=(L(T_n,n-y))_{0\leqslant y\leqslant n}$$
has the same distribution as $(Z_y)_{0\leqslant y\leqslant n}$ where $(Z_y)_{y \in \mathbb{N}}$ is a branching process with immigration in random environment (BPIRE) defined by
\begin{equation}
\label{eq:definition-Z-y^n}
Z_0=0
\quad \text{and} \quad
\forall y \in \mathbb{N}^* , \, Z_{y}=\sum_{i=0}^{Z_{y-1}}\xi_{y,i},
\end{equation}
where $\parenth{\xi_{y,i}}_{(y,i)\in\relatif_+^*\times\relatif_+}$ is a family of geometrically distributed independent random variables, \ie\ for all $(y,i)\in\relatif_+^*\times\relatif_+$ and for all $k\in\relatif_+$,
$$
\Pbf_\omega\parenth{\xi_{y,i}=k}=\omega_{y}(1-\omega_{y})^k\enspace.
$$
Under the annealed law $\Pbf^\nu$,  $(Z_x)_{x\in\relatif_+}$ is an homogeneous Markov chain starting at 0 with transition kernel \cite[Proposition 4.3]{comets:falconnet:2014}
\begin{align}\label{eq:ker1}
K^\nu(i,j)=\binom{i+j}{i}\int_0^1a^{i+1}(1-a)^{j}\nu(\rmd a)\eqsp.
\end{align}

\section{Estimator construction}
\label{sec:estimator-construction}
%
All along the paper, we assume that $\nu$ is absolutely continuous and has density $\tdens$ with respect to the Lebesgue measure.\

Let $\beta>0$, $L>0$ and $m= \sup\set{\ell \in \relatif_+}{\ell < \beta}$. The set $\Sigma(\beta,L)$ is the set of $m$ times differentiable functions $g : \intff{0}{1} \mapsto \reelp$ satisfying for all $(x,x') \in \intff{0}{1}^2$
\begin{equation}\label{eq:def-Holder}
\abs{g^{(m)}(x) - g^{(m)}(x')}
\leqslant L \abs{x-x'}^{\beta-m} \eqsp.
\end{equation}
The following section provides the construction of basic density estimators.

As suggested by \eqref{eq:ker1}, the moments of $\nu$ provide the natural information available for our estimation problem. Recovering a density function from its moments is a classical instance of the Hausdorff  moment problem; see \cite{akhiezer:1965}.

For any $j \in \relatif_+$, let $\mu_{j}[\nu]$ denote the $j$-th moment of the random environment distribution
\[
\mu_{j}[\nu] = \int_0^1 t^j \nu(\rmd t)=\int_0^1t^jf(t)\rmd t
 \eqsp.
\]
The moment problem is said to have unique solution if for two distributions $\nu$ and $\nu'$ over $\intoo{0}{1}$, the equations
$\mu_j[\nu]= \mu_j[\nu']$ for any $j\geqslant 1$ imply that $\nu= \nu'$. If the distribution $\nu$ satisfies this property, it is said to be moment determinate. It is shown that a sufficient condition for $\nu$ to be moment determinate is that the distribution $\nu$ on $(0,1)$ has a continuous density; see \cite{mnastakanov:ruymgaart:2003,mnatsakanov:2008}.

When $\nu$ is moment determinates, the operator mapping the c.d.f. of $\nu$ into $\po\mu_{j}[\nu]\pf_{j \in \relatif_+}$ is therefore invertible. Given $\param\in\relatif_+$, it has been shown in \cite{mnatsakanov:2008} that  its inverse can be approximated by
\begin{equation}\label{eq:DefKM}
\nu_M([0,x])= \sum_{k=0}^{\lfloor \param x\rfloor} \sum_{j=k}^\param \binom{\param}{j}\binom{j}{k} (-1)^{j-k} \mu_{j}[\nu] \eqsp.
\end{equation}
It is shown in \cite{mnatsakanov:2008} that $\lim_{M \to \infty}\nu_M([0,x]) = \nu([0,x])$ at every point $x \in (0,1)$ where the c.d.f. is continuous.
If $\nu$ has a density $f$, we consider then the following approximation, defined for any $x\in[0,1]$ by
\begin{equation}
\label{eq:estim-density}
\tdens^{\param}(x) = (\param+1) \left(\nu_\param\left(\left[0,x+\param^{-1}\right]\right)-\nu_\param\left(\left[0,x\right]\right)\right) \eqsp .
\end{equation}
\cite[Theorem~1]{mnatsakanov:2008} ensures that, when $f$ is continuous, $\tdens^{\param}$ converges uniformly to $\tdens$. More precisely, we establish in \cref{lem:ControleBias} that if $\tdens \in \Sigma(\beta,L)$ with $\beta \in \intof{0}{2}$ then for any $x \in \intff{0}{1}$
$$
\abs{\tdens(x) - \adens(x)} \leqslant \cstholder \param^{-\beta/2} \eqsp.
$$
In the following, we do not use the expression \eqref{eq:DefKM} to construct our estimator.
For any $(a,b)$ in $\reel_+^*\times\reel_+^*$,
benote by $\targetmoment^{a,b}$ the following beta-moments
\begin{equation}
 \label{eq:def-betamoment}
\targetmoment^{a,b} = \Ebb^\nu\crochet{\omega_0^a(1-\omega_0)^b} = \int_0^1 u^a (1-u)^b \tdens(u) \rmd u \eqsp
\end{equation}
where  $\betaker{a}{b}$ is the beta-kernel defined in \eqref{def:BetaKer}.
Using these notations, the approximation $\tdens^\param$ of the density may be written \cite[Eq~(12)]{mnatsakanov:2008}
\begin{align}
\label{eq:expr-approx-density}
\adens(x)
&= \int_0^1 \betaker{\lfloor \param x \rfloor + 1}{\param-\lfloor \param x \rfloor + 1}(u) \tdens(u) \rmd u
= (\param+1) \binom{\param}{\pent{\param x}} \targetmoment^{\pent{\param x},\param-\pent{\param x}} \eqsp .
\end{align}
When $a$ and $b$ are integers, the moments $\targetmoment^{a,b}$ can be consistently estimated as shown in \cite{diel:lerasle:2016}. Using the convention $0/0=0$, denote for all $(i,j) \in \relatif_+$
\begin{gather}\label{eq:def-phi}
\phi^{a,b}(i,j) = \frac{\binom{i+j-(a+b)}{i-a}}{\binom{i+j}{i}} \indic{\{i\geqslant a, j\geqslant b\}}
\quad  \text{ and } \quad
N_n^a = \sumforce{\etat=1}{n} \indic{\{Z_{\etat-1}^n\geqslant a\}} \eqsp ,\\
\label{eq:def-estim-moment}
\momentestim_n^{a,b} = \frac{1}{N_n^a}\sumforce{\etat=1}{n}\phi^{a,b}(Z_{\etat-1}^n,Z_\etat^n)\eqsp,
\end{gather}
where $(Z_\etat^n)$ is defined in \eqref{eq:definition-Z-y^n}.
Therefore, an estimator of $f(x)$ might be obtained by plugging an estimator of the beta moment in \eqref{eq:expr-approx-density}.
\begin{equation}\label{eq:def-estim-density}
\forall x \in \intff{0}{1},\qquad \edens_n^\param(x) =
(\param+1) \binom{\param}{\pent{\param x}} \momentestim_n^{\pent{\param x},\param-\pent{\param x}} \eqsp .
\end{equation}
Statistical properties of these estimators are provided in Theorem~\ref{Thm:NonAdaptive}.

\section{Main results}\label{sec:MainRes}
Let us first derive convergence rates of the basic estimators $\edens_n^\param$ for fixed values of the regularization parameter $\param$.

\begin{theorem}\label{Thm:NonAdaptive}
Assume that the distribution $\nu$ of the random environment  has a density $\tdens\in\Sigma(\beta,L)$ for some $\beta\in\intof{0}{2}$ and $L\geqslant 1$.
\begin{enumerate}[(i)]
\item Assume \Cref{assum:recurrent}. For any $M\geqslant 1$, there exists a constant $C_\nu$ such that
\begin{equation}\label{eq:NABRecurrent}
\Ebf^\nu\crochet{\norm{\infty}{\tdens - \edens_n^{M}}}
\leqslant C_{\nu}\parenth{L\param^{-\beta/2} +
M\frac{\log n}{\sqrt{n}}}\eqsp,
\end{equation}
\item Assume \Cref{assum:transient-to-right}. For any $M\geqslant 1$
\begin{equation}\label{eq:NABTransient}
\Ebf^\nu\crochet{\norm{\infty}{\tdens - \edens_n^{M}}}
\leqslant C_{\nu}\parenth{L\param^{-\beta/2} +
M^{1+\kappa}\sqrt{\frac{\log n}{n}}}\eqsp
\end{equation}
where $\kappa$ is defined in \Cref{assum:transient-to-right}.
\end{enumerate}
\end{theorem}
Theorem~\ref{Thm:NonAdaptive} provides the first rates of convergence for  non-parametric density estimation of the random environment of a random walk. An important feature is that these rates are achieved by the same estimator in the recurrent and transient regimes. On the other hand, the value $\param_n$ optimizing these rates highly depends on both the regime of the walk and the regularity $\beta$ of $f$ that are both unknown in practice.
\begin{corollary}
\begin{enumerate}[(i)]
\item Assume \Cref{assum:recurrent}. Then, taking $\param_n = \lfloor \po\sqrt{n}/\log(n) \pf^{2/(\beta+2)} \rfloor$,
$$
\Ebf^\nu\crochet{\norm{\infty}{\tdens - \edens_n^{\param_n}}}
\leqslant
C_{\nu} \po \frac{\log(n)}{\sqrt{n}} \pf^{\beta/(\beta + 2)}\eqsp .
$$
\item Assume \Cref{assum:transient-to-right}. Then, taking $\param_n =  \po \lfloor n / \log(n) \pf^{\frac1{\beta + 2(1+\kappa)}} \rfloor$,,
$$
\Ebf^\nu\crochet{\norm{\infty}{\tdens - \edens_n^{\param_n}}}
\leqslant
C_{\nu}\po \frac{\log(n)}{n} \pf^{\frac{ \beta}{2(\beta + 2(1+\kappa))}} \eqsp.
$$
\end{enumerate}
\end{corollary}
We will next propose a method to construct  an estimator which is adaptive to the regime (recurrence / transience) and the regularity of the density of the random environment. To that purpose, we follow the Goldenshluger-Lepski strategy \cite{goldenshluger:lepski:2008} to obtain a data-driven choice of the beta-moment order $\lepskiparam_n$.
In the proof of Theorem~\ref{Thm:NonAdaptive}, we build confidence regions for $\tdens$. More precisely, it follows from Eq.~\eqref{eq:DecBasique} and Lemmas~\ref{lem:ControleBias},~\ref{lem:ForLepski} that there exists an event $\Omega_0$ with probability larger than $1-n^{-2}$, such that for all $M\in\{1,\ldots,n-1\}$,
\begin{equation}
\label{eq:intervalle-confiance}
\norm{\infty}{\tdens - \edens_n^\param}\leqslant
31L\param^{-\beta/2} +
\frac{M+1}{N_n^{\param}}\sqrt{3n\log n}\eqsp .
\end{equation}
Define, for any $M\in \{1,\ldots,n-1\}$,
\begin{align*}
B(M)&=31L\param^{-\beta/2} \quad V_n(M)=\frac{M+1}{N_n^{\param}}\sqrt{3n\log n}\eqsp,\\
C_n(M)&=\sup_{M'\in \{1,\ldots,n-1\}}\accol{\norm{\infty}{\edens_n^M-\edens_n^{M\vee M'}}-2V_n(M')}\eqsp.
\end{align*}
Goldenshluger-Lepski's choice of $M$ is defined by
\begin{equation}\label{eq:GLEst}
 \lepskiparam_n\in \argmin{M\in \{1,\ldots,n-1\}}{C_n(M)+2V_n(M)}\eqsp.
\end{equation}
%
The estimator satisfies the following lemma, whose proof is given for completeness.
\begin{lemma}\label{lem:GL}
On the event $\Omega_{0}$, for any $M\in \{1,\ldots,n-1\}$
\[
\norm{\infty}{\tdens-\edens_n^{\lepskiparam_n}}\leqslant 5B(M)+5V_n(M)\eqsp.
\]
\end{lemma}
\begin{proof}
Denote $d(\tdens,g)= \norm{\infty}{f-g}$.
By repeated applications of the triangular inequality, one obtains, for any $M\in \{1,\ldots,n-1\}$, on $\Omega_0$,
\begin{align*}
 d(\tdens,\edens_n^{{\lepskiparam_n}})&\leqslant d(\tdens,\edens_n^{M})+d(\edens_n^{M},\edens_n^{M\vee{\lepskiparam_n}})+
 d(\edens_n^{M\vee{\lepskiparam_n}},\edens_n^{{\lepskiparam_n}})\\
 &\leqslant B(M)+V_n(M)+C_n(M)+2V_n({\lepskiparam_n})+C_n({\lepskiparam_n})+2V_n(M)\\
 &\leqslant B(M)+V_n(M)+C_n(M)+2V_n(M)+C_n(M)+2V_n(M)\\
 &\leqslant B(M)+5V_n(M)+2C_n(M)\enspace.
\end{align*}
Now, for any $M\in \{1,\ldots,n-1\}$, $B$ being non increasing and $V_n$ being non decreasing, on $\Omega_0$
\begin{align*}
 C_n(M)&\leqslant 0\vee \sup_{M'>M}\accol{d(\tdens,\edens_n^M)+d(\tdens,\edens_n^{M'})-2V_n(M')}\\
 &\leqslant0\vee \sup_{M'>M}\accol{ B(M)+B(M')+V_n(M)-V_n(M')}\leqslant 2B(M)\enspace.
\end{align*}
\end{proof}
Using \Cref{lem:GL} together with \eqref{eq:intervalle-confiance} yields the following Theorem:
\begin{theorem}\label{Thm:Adaptive}
Let $\beta\in\intof{0}{2}$ and $L\geqslant 1$.
Assume that $\tdens\in\Sigma(\beta,L)$ holds.
\begin{enumerate}[(i)]
\item Assume \Cref{assum:recurrent}. Then, there exists a constant $C_\nu$ such that
$$
\Pbf^\nu\parenth{\norm{\infty}{\tdens - \edens_n^{\lepskiparam_n}}
\leqslant
C_{\nu} \po \frac{\log(n)}{\sqrt{n}} \pf^{\beta/(\beta + 2)}}\geqslant 1-C_\nu\frac{\log n}{\sqrt n}\eqsp .
$$
\item Assume \Cref{assum:transient-to-right}. Then, there exists a constant $C_\nu$ such that
$$
\Ebf^\nu\crochet{{\norm{\infty}{\tdens - \edens_n^{\lepskiparam_n}}}}
\leqslant
C_{\nu}\po \frac{\log(n)}{n} \pf^{\frac{\beta}{2(\beta + 2(1+\kappa))}}\eqsp ,
$$
\end{enumerate}
\end{theorem}
The estimator $\edens_n^{\lepskiparam_n}$ achieves simultaneously the optimal rates of Theorem~\ref{Thm:NonAdaptive} in every regime. It is in this sense adaptive to the regime and the regularity of $\tdens$.

The result holds in expectation in transient regimes and only with large probability in the recurrent regime. The reason is that the term $N_n^M$ cannot be handled in the same way in transient regimes where the Markov chain $(Z_k)$ admits a unique invariant probability and in recurrent regimes where it does not.

The rates do not match those one could reach if the environment $\omega$ is observed \cite{bertin:klutchnikoff:2010}. When the chain is recurrent, it visits typically more than $n$ times almost all sites  \cite[Lemma 10]{diel:lerasle:2016}. Therefore, one can probably estimate the environment itself and plug these estimators in any standard density estimator. This might improve the rates in the recurrent regime but it would certainly provide a poor estimator in transient regimes. The problem of recovering the distribution of the environment is in general harder if one observes only one trajectory of the walk than if the environment is observed directly. It is not clear whether a single estimator can recover minimax rates of the i.i.d. setting in the recurrent regime and our rates in transient regimes.


\section{Simulation Study}\label{sec:SimStud}

In this section, we illustrate our theoretical results. We  consider RWRE in different regimes (recurrent, transient ballistic and sub-ballistic). We also investigate the impact of the regularity of the density of the environment
The number of steps of the random walk that are observed depend on the regime. It follows from our discussion in \Cref{sec:setting}, that if the RWRE is transient to the right ($\Ebb^\nu[\log \rho_0] < 0$ where $\rho_0$ is defined in \eqref{eq:definition-rho-0}) then either
\begin{enumerate}
\item $T_n/n \to (1 + \Ebb^\nu[\rho_0])/(1-\Ebb^\nu[\rho_0])$ $\Pbf^\nu$-a.s. in the ballistic regime $\Ebb^\nu[\rho_0] < 1$, in which case the number of time steps grow linearly with $n$
\item $T_n/n \to \infty $ converges $\Pbf^\nu$-a.s.  in the sub-ballistic regime $\Ebb^\nu[\rho_0] \geq 1$, showing that the number of time steps grow superlinearly with $n$. In the sub-ballistic case, the fluctuations of $T_n$ are controlled by the parameter $\kappa$ defined in \eqref{eq:kappa}.
\end{enumerate}

We then compare the performance of our estimator in three sub-ballistic regime but with different $\kappa$. In both parts, we consider our estimation procedure for different values of the parameter $M$.

\subsection{Influence of the regularity}
In this section, we consider 3 different distributions for the environment.
\begin{itemize}
 \item The first example is the Beta distribution $B(3,3)$ whose density $\betaker{3}{3}$ is defined in \eqref{def:BetaKer}.
 \item The second example is the distribution with density
\begin{equation*}
x \mapsto 4 \left(1-\frac{\abs{2x-0.6}}{0.15}\right) \indic{\intff{0.225}{0.325}}(x)
+ \frac{14}{3} \left(1-\frac{\abs{2x-1.2}}{0.3}\right) \indic{\intff{0.45}{0.75}}(x)\eqsp .
\end{equation*}
\item The last example is the distribution with density
\begin{equation*}
x \mapsto 0.6 \betaker{3}{3}(2x) \indic{\ccint{0,0.5}}(x) + 1.4 \betaker{3}{3}(2x-1) \indic{\ocint{0.5,1}}(x)\eqsp .
\end{equation*}
\end{itemize}


These densities are plotted in black in Figure~\ref{BoxplotRegimes}. In the first example, the chain is recurrent, in the second and the third, it is sub-ballistic with a value of $\kappa$ respectively equal to $0.17$ and $0.57$.

\medskip

For each distribution of the environment, $100$ trajectories of the RWRE have been simulated until the chain reaches $100$.

In each case, our (piecewise constant) estimator $\edens_n^M$, defined in \eqref{eq:def-estim-density}, built for $M\in\{25,50,75\}$ is compared with the actual density $\tdens$.
More precisely, figure \ref{BoxplotRegimes} shows for each considered environment and for $\param$ in $\{25,50,75\}$ a graphic representation of the true density (solid black line) and in each point $x \in \intoo{0}{1}$ the median of the estimator $\edens_n^M(x)$ (dashed-dotted blue line), its interquartile range (greyed area) and its lower and upper hinges (dotted blue line).

According to our theoretical bounds, the estimator $\edens_n^M$ performs better when the true density is regular as in Examples 1 and 3 than in the non regular Example 2. The estimator $\edens_n^M$, based on a polynomial approximation, cannot properly approximate the two triangles shaped density of Example 2. More precisely, the regularization parameter $M$ should probably be much larger for the bias of the estimator to be small.

\begin{figure}
\includegraphics[scale=0.33]{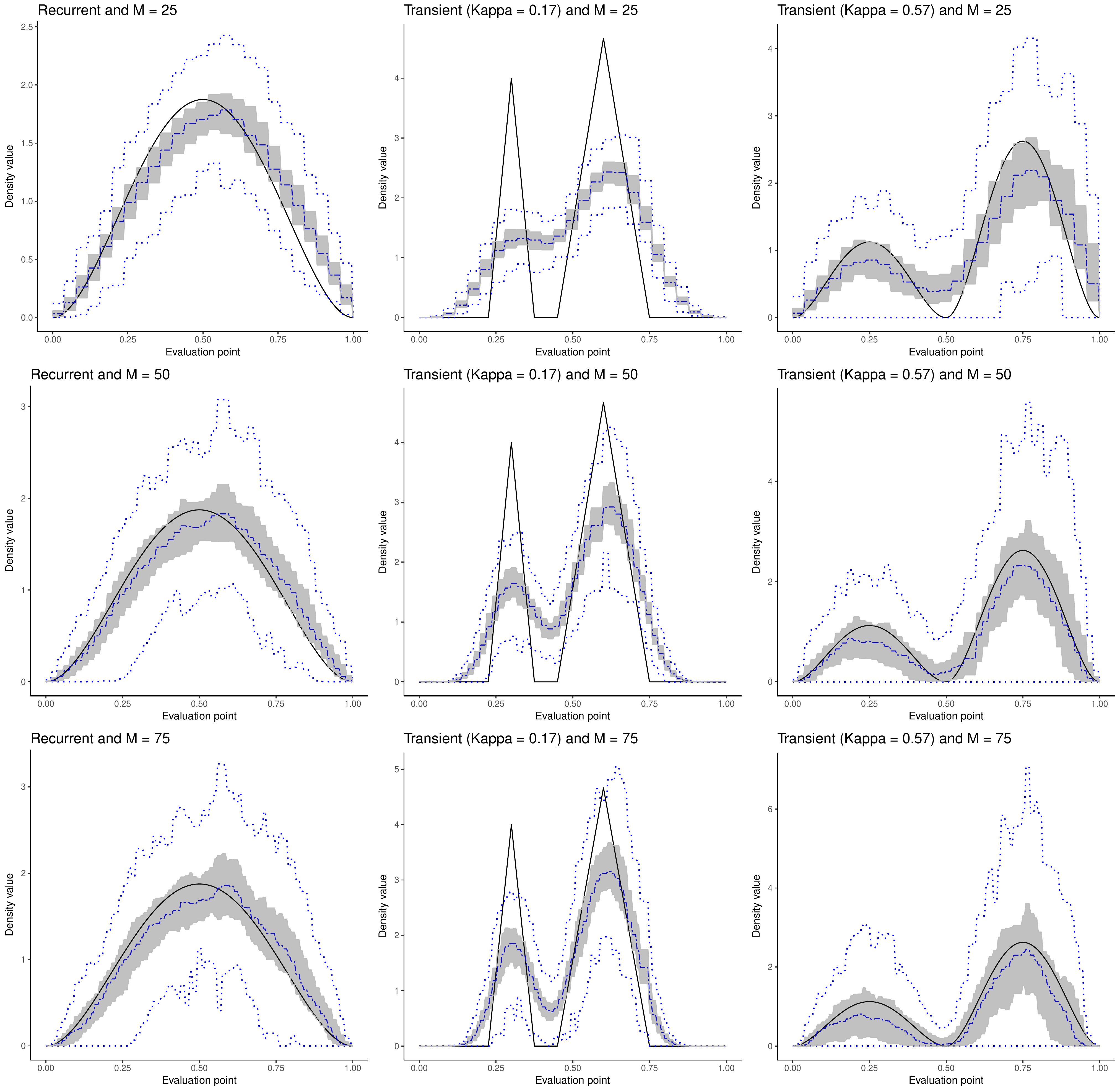}
\caption{\label{BoxplotRegimes}
True density (solid black line) and in each point $x \in \intoo{0}{1}$ the median of the estimator $\edens_n^M(x)$ (dashed-dotted blue line), its interquartile range (greyed area) and its lower and upper hinges (dotted blue line) of the estimator $\edens_{100}^\param$ for the three considered environment (one environment per column) and for $M\in \{25,50,75\}$.}
\end{figure}

\subsection{Influence of the regime}
We now consider trajectories from RWRE that all are transient to right and sub-ballistic.
The environment is drawn  from the distribution with density
\begin{equation}\label{def:density-bossestele}
x \mapsto 0.8 \betaker{3}{3}(0.5+2(x-c_1)) + 1.2 \betaker{3}{3}(0.5+2(x-c_2))
\end{equation}
and for three set of the parameters $c_1$ and $c_2$ such that the parameter $\kappa$ associated to the sub-ballistic RWRE takes values near the extreme values of this regime. In the following table, the chosen sets of parameters are provided with the corresponding (numerically approximated) values of $\kappa$.
\[
\begin{tabular}{|c|c|c|c|c|c|c|}
\hline
$\kappa$ & 0.04
& 0.35
& 1\\
\hline
$c_1$ & 0.27 & 0.3 & 0.38\\
\hline
$c_2$ & 0.67 & 0.7 & 0.7\\
\hline
\end{tabular}
\]

For each set of parameters for the environment distribution, 100 trajectories until the RWRE reaches $100$ for the first time have been generated.

Figure~\ref{fig:boxplot-estim-bossestele} illustrates the variability of the estimator $\edens_{100}^\param$. For the densities defined in~\ref{def:density-bossestele} with $\kappa$ in $\{0.04, 0.35, 1\}$, we plot the true density of the environment as well as
the interquartile range and lower and upper hinges of $\edens_{100}^M$  computed over the $100$ trajectories and for $M\in\{25,50,75\}$.\\

As expected from our theoretical results, the variability of the estimator deteriorates 
with $\kappa$. 
When $\kappa = 0.04$ or $\kappa=0.35$, the estimator shows better bias and worse variability when $\param$ grows. The situation in the nearly ballistic case ($\kappa=1$) is different : the walk reaches faster $n=100$ and the trajectory is thus shorter, one single trajectory reaching $n=100$ does not seem sufficient to obtain a reliable estimator of $\tdens$.

\begin{figure}
\includegraphics[scale=0.33]{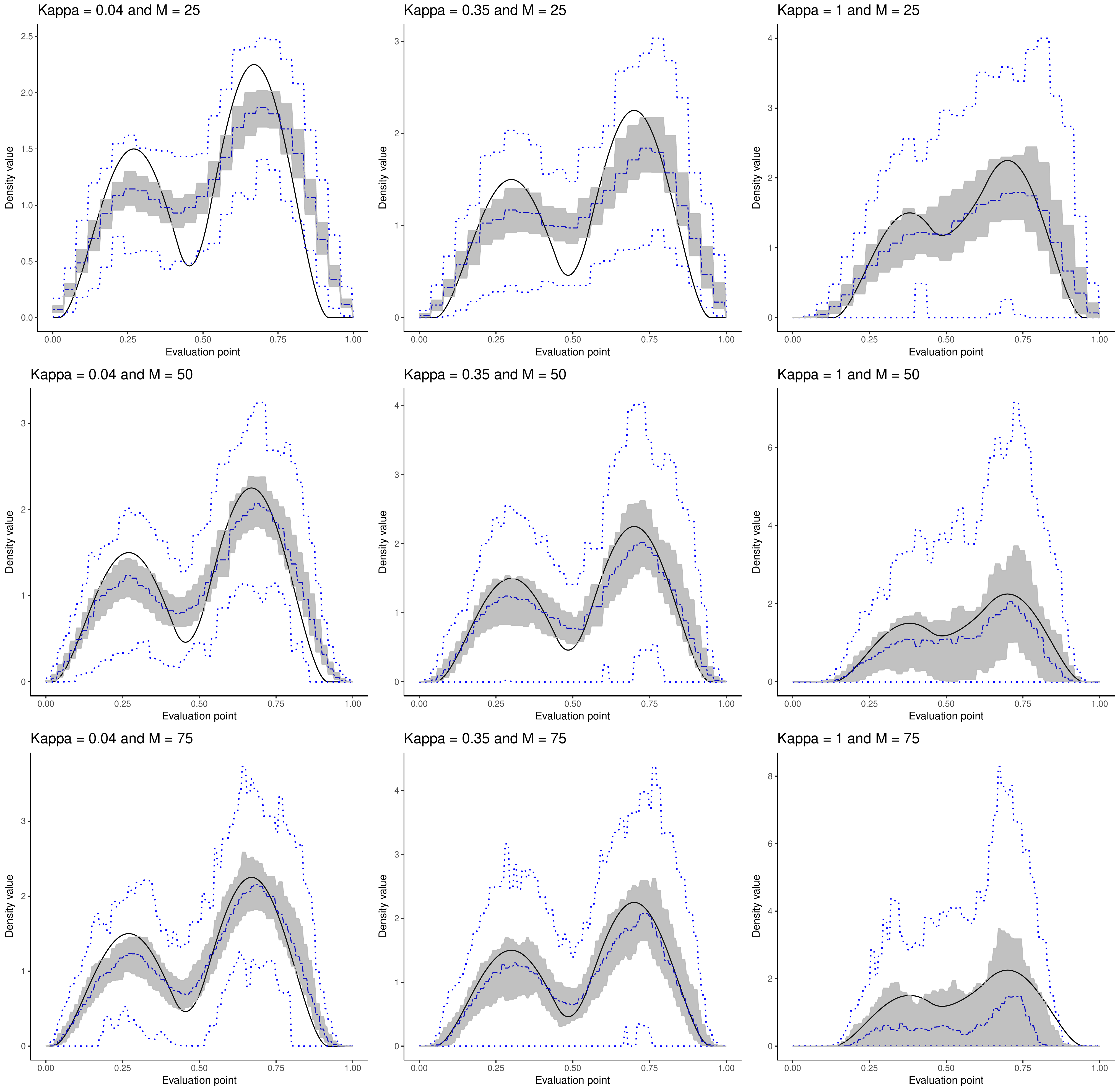}
\caption{\label{fig:boxplot-estim-bossestele}
True density (in black) and boxplot (median, first and third quartiles as well as lower and upper hinge) of the estimator $\edens_{100}^\param$ for three different values of parameter $\param$ ($25$, $50$ and $75$) and $\kappa$ ($0.04$, $0.35$ and $1$)
}
\end{figure}

\section{Proof}\label{sec:Proof}

\subsection{Proof of Theorem~\ref{Thm:NonAdaptive}}
By the triangular inequality
\begin{align}\label{eq:DecBasique}
\norm{\infty}{\tdens - \edens_n^\param}
& \leqslant
\norm{\infty}{\adens - \tdens} +\norm{\infty}{\adens - \edens_n^\param}\eqsp .
\end{align}
The first term in the right-hand side is deterministic and called the bias, the second is stochastic and called the variance. Lemma \ref{lem:ControleBias} provides an upper bound for the bias. Random bounds on the variance, valid in any regime, are provided in Lemma~\ref{lem:ForLepski}.  Lemma~\ref{lem:VarianceBasique} uses these random bounds to derive deterministic bounds on the variance that, once plugged in \eqref{eq:DecBasique} together with the result of Lemma \ref{lem:ControleBias} allow to conclude the proof of \eqref{eq:NABTransient} and \eqref{eq:NABRecurrent}. The choice of $\param_n$ optimizing the bounds is then immediate.
\begin{lemma}\label{lem:ControleBias}
Let $\beta\in\intof{0}{2}$ and $L\geqslant 1$. Assume that$\tdens\in\Sigma(\beta,L)$.  Then
\begin{equation}\label{eq:ControleBias}
\norm{\infty}{\tdens - \adens} \leqslant 31L\param^{-\beta/2} \eqsp.
\end{equation}
\end{lemma}

\begin{proof}
Let $x\in[0,1]$ and let $B_{(\param,x)}$ be a beta distributed random variable with parameters $\pent{\param x}+1$ and $\param-\pent{\param x}+1$ on a probability space $\left(\Omega,\mathcal{A},\Pbb\right)$. By definition,
\begin{align*}
\adens(x) - \tdens(x)
& =
\int_0^1 \tdens(u) \betaker{\pent{\param x}+1}{\param-\pent{\param x}+1}(u) \mathrm{d}u - \tdens(x)
=
\esp{\tdens(B_{(\param,x)}) - \tdens(x)} \eqsp.
\end{align*}
%
Consider first the case $0 < \beta \leqslant 1$.
Since $f$ is $(\beta,L)$-Hölder, by Jensen's inequality,
$$
\abs{\adens(x) - \tdens(x)}
\leqslant
\esp{\abs{\tdens(B_{(\param,x)}) - \tdens(x)}}
\leqslant
L\esp{\abs{B_{(\param,x)} - x}^\beta}
\leqslant
L\esp{\abs{B_{(\param,x)} - x}}^\beta \eqsp .
$$
For any $\eta > 0$, it holds
\begin{equation}\label{eq:BoundBias1}
 \esp{\abs{B_{(\param,x)} - x}}
\leqslant
\eta + \esp{\abs{B_{(\param,x)} - x}\indic{\{\abs{B_{(\param,x)} - x}>\eta\}}} \eqsp .
\end{equation}
By Cauchy-Schwarz and Markov inequalities,
\begin{align}
\notag\esp{\abs{B_{(\param,x)} - x}\indic{\{\abs{B_{(\param,x)} - x}>\eta\}}}&
\leqslant \sqrt{\esp{\parenth{B_{(\param,x)} - x}^2}\prob{\abs{B_{(\param,x)} - x}>\eta\}}}\\
\label{eq:inequality-CauchyMarkovBias}
&\leqslant
\sqrt{\po \var{B_{(\param,x)}} + \po \esp{B_{(\param,x)}} - x \pf^2 \pf \frac{\var{B_{(\param,x)}}}{\eta^2}}\eqsp .
\end{align}
One has
\begin{equation*}
\esp{B_{(\param,x)}}
=
\frac{\pent{\param x}+1}{\param+2} \quad \text{ and } \quad \var{B_{(\param,x)}}
=
\frac{(\pent{\param x}+1)(M-\pent{\param x}+1)}{(\param+2)^2(\param+3)}
\eqsp \eqsp .
\end{equation*}
Therefore,
\begin{equation}\label{eq:inequality-esp-beta}
\abs{\esp{B_{(\param,x)}} - x} \leqslant \frac{2}{M+2} \quad \text{ and } \quad \var{B_{(\param,x)}}
\leqslant
\frac{1}{4 \param} \eqsp .
\end{equation}
Plugging \eqref{eq:inequality-esp-beta} in \eqref{eq:inequality-CauchyMarkovBias} yields:
\begin{align*}
\esp{\abs{B_{(\param,x)} - x}\indic{\{\abs{B_{(\param,x)} - x}>\eta\}}}
\leqslant
\sqrt{\po \frac{1}{4\param} + \frac{4}{\param^2} \pf \frac{1}{4 M \eta^2}}
\leqslant
\frac{1}{\eta} \frac{\sqrt{17}}{4 M} \eqsp .
\end{align*}
Plugging this bound in \eqref{eq:BoundBias1} and optimizing in $\eta > 0$ yields
$$
\abs{\adens(x) - \tdens(x)}
\leqslant
17^{\beta/4} L \, \param^{-\beta/2} \eqsp .
$$

Consider then the case $1 < \beta \leqslant 2$. Since $\tdens$ is differentiable, by the mean-value theorem, there exists a random variable $\tilde{B}_{(\param,x)}$ such that
$$
\abs{\tilde{B}_{(\param,x)} - x} \leqslant \abs{B_{(\param,x)} - x}
\quad \text{and} \quad
\tdens(B_{(\param,x)}) - \tdens(x)=f'(\tilde{B}_{(\param,x)}) (B_{(\param,x)} - x)\enspace.
$$
Therefore
\begin{align*}
\abs{\adens(x) - \tdens(x)}
& = \abs{\esp{f'(x) (B_{(\param,x)} - x) + \po f'(\tilde{B}_{(\param,x)})-f'(x)\pf (B_{(\param,x)} - x)}}\\
& \leqslant
\abs{f'(x)} \abs{\esp{B_{(\param,x)}} - x} + \esp{\abs{f'(\tilde{B}_{(\param,x)})-f'(x)} \abs{B_{(\param,x)} - x}} \eqsp .
\end{align*}
Using that $f$ is $(\beta,L)$-Hölder and \eqref{eq:inequality-esp-beta}, one gets
$$
\abs{\adens(x) - \tdens(x)}
\leqslant
\norm{\infty}{f'} \frac{2}{M} + L\esp{\abs{B_{(\param,x)} - x}^\beta} \eqsp .
$$
Furthermore,
\begin{align*}
\esp{\abs{B_{(\param,x)} - x}^\beta}
& \leqslant
\esp{\abs{B_{(\param,x)} -\esp{B_{(\param,x)}} +\esp{B_{(\param,x)}} - x}^\beta}\\
& \leqslant
2^{\beta-1} \po \esp{\abs{B_{(\param,x)} -\esp{B_{(\param,x)}}}^\beta} +\abs{\esp{B_{(\param,x)}} - x}^\beta \pf \eqsp .
\end{align*}
By \eqref{eq:inequality-esp-beta},
\begin{align*}
\esp{\abs{B_{(\param,x)} - x}^\beta}
&\leqslant
2^{\beta-1} \po \var{B_{(\param,x)}}^{\beta/2} +\abs{\esp{B_{(\param,x)}} - x}^\beta \pf\\
& \leqslant
2^{\beta-1} \po \frac{2^{\beta/2}}{\param^{\beta/2}} +\frac{2^\beta}{\param^\beta} \pf
\leqslant
\frac{4^\beta}{\param^{\beta/2}} \eqsp .
\end{align*}
By \cite[Theorem~1.1]{tsybakov:2009},
$$
\norm{\infty}{f'}\leqslant \frac{3}{\sqrt{4-\sqrt{13}}} \; \frac{L+\beta+1}{\beta+1} \leqslant 7.2 L\eqsp .
$$
It thus finally holds:
$$
\abs{\adens(x) - \tdens(x)}
\leqslant
31L \param^{-\beta/2}\eqsp.
$$
As this result holds for all $x\in[0,1]$, the proof is complete.
\end{proof}
Let us now turn to the random part of the risk. We start with a random upper bound.
\begin{lemma}\label{lem:ForLepski}
 Assuming \Cref{assum:transient-to-right} or \Cref{assum:recurrent} , then, for any $z>0$,
  \[
\Pbf^\nu\parenth{\norm{\infty}{\adens - \edens_n^\param }>\frac{M+1}{N_n^{\param}}\sqrt{\frac{n(z+\log \param)}{2}}}\leqslant 2e^{-z}\enspace.
 \]

\end{lemma}
\begin{proof}
By definition, one has
\[
\adens(x) - \edens_n^\param(x)=(\param+1) \binom{\param}{\pent{\param x}}\parenth{ \targetmoment^{\pent{\param x},\param-\pent{\param x}}- \momentestim_n^{\pent{\param x},\param-\pent{\param x}} }
\]
 By \cite[Theorem 4]{diel:lerasle:2016}, for any $z>0$,
 \[
\Pbf^\nu\parenth{\abs{ \targetmoment^{\pent{\param x},\param-\pent{\param x}}- \momentestim_n^{\pent{\param x},\param-\pent{\param x}} }>\frac{n}{N_n^{\pent{\param x}}}\binom{\param}{\pent{\param x}}^{-1}\sqrt{\frac{z}{2n}}}\leqslant 2e^{-z}\enspace.
 \]
It remains to remark that $N_n^{\pent{\param x}}\geqslant N_n^{\param}$ and to apply a union bound to conclude the proof.
\end{proof}
The following lemma provides deterministic rates of convergence that derive from Lemma~\ref{lem:ForLepski}.
\begin{lemma}\label{lem:VarianceBasique}
Under \Cref{assum:transient-to-right}, $X$ is transient and
  \[
\Ebf^\nu\crochet{\norm{\infty}{\adens - \edens_n^\param }}\leqslant C_{\nu}M^{1+\kappa} \sqrt{\frac{\log n}{n}}\eqsp.
 \]
Under \Cref{assum:recurrent}, $X$ is recurrent and
  \[
\Ebf^\nu\crochet{\norm{\infty}{\adens - \edens_n^\param }}\leqslant C_{\nu}M\frac{\log n}{\sqrt{n}}\eqsp.
 \]
%
\end{lemma}
\begin{proof}
Fix some $M\geqslant 1$.
If the chain is transient, by \cite[Lemmas 8 and 9]{diel:lerasle:2016}
, there exists a constant $C_\nu$ such that, for any $z>0$,
\[
\Pbf^\nu\parenth{\frac{N_n^{\param}}n\geqslant \frac{1}{C_{\nu}}\parenth{\frac1{ M^\kappa}-\frac{C_\nu\sqrt{z}}{\sqrt{n}}}}\geqslant 1-2e^{-z }\enspace.
\]

Therefore,
\begin{equation}\label{eq:BoundNnM}
\Pbf^\nu\parenth{\frac{N_n^{\param}}n\geqslant \frac{1}{C_{\nu} M^{\kappa}}} \geqslant1-2e^{-n/C_\nu^2M^{2\kappa}} \eqsp.
\end{equation}
Combining \eqref{eq:BoundNnM} with Lemma~\ref{lem:ForLepski} yields, for any $z\in\intof{0}{n/4C_\nu^2M^{2\kappa}}$,
\begin{equation}\label{eq:probBound}
\Pbf^\nu\parenth{\norm{\infty}{\adens - \edens_n^\param }>C_{\nu} M^{(1+\kappa)}\sqrt{\frac{z+\log M}{n}}}\leqslant 4e^{-z}\enspace.
\end{equation}
By \eqref{eq:def-phi} and \eqref{eq:def-estim-moment}, both $\adens$ and $\edens_n^\param$ are uniformly upper-bounded by $\param + 1$. One can therefore integrate \eqref{eq:probBound} to get, for any $M\leqslant n^{1/2\kappa}$,
\[
\Ebf^\nu\crochet{\norm{\infty}{\adens - \edens_n^\param }}\leqslant C_{\nu}M^{1+\kappa} \sqrt{\frac{\log M}{n}}\eqsp.
\]
As the bound also holds obviously for $M\geqslant n^{1/2\kappa}$ the proof of the transient case is complete.

Assume now that the chain is recurrent, then \cite[Lemma 10]{diel:lerasle:2016} states that there exists $C_{\nu}$ such that,
\begin{equation}\label{eq:NnMRec}
\Pbf^\nu\parenth{N_n^{n}\leqslant \frac n2}\leqslant C_{\nu}\frac{\log n}{\sqrt n}\eqsp.
\end{equation}
Combining \eqref{eq:NnMRec} with that of Lemma~\ref{lem:ForLepski} implies that, for any $M\in\{1,\ldots,n\}$,
  \[
\Pbf^\nu\parenth{\norm{\infty}{\adens - \edens_n^\param }>(M+1)\sqrt{\frac{2(z+\log \param)}{n}}}\leqslant 2e^{-z}+C_{\nu}\frac{\log n}{\sqrt n}\eqsp.
 \]
Choosing $z=\log n$ and noticing that $\norm{\infty}{\adens}\vee \norm{\infty}{\edens_n^\param}\leqslant M+1$ concludes the proof.
\end{proof}
\paragraph{Conclusion of the proof of Theorem~\ref{Thm:NonAdaptive}} Plugging the conclusions of Lemmas~\ref{lem:ControleBias} and ~\ref{lem:VarianceBasique} in \eqref{eq:DecBasique} yields, if $X$ is transient,
\begin{align*}
\Ebf^\nu\crochet{\norm{\infty}{\tdens - \edens_n^\param}}
& \leqslant
31L\param^{-\beta/2} +
C_{\nu} M^{1+\kappa}\sqrt{\frac{\log n}{n}}\eqsp .
\end{align*}
If $X$ is recurrent,
\begin{align*}
\Ebf^\nu\crochet{\norm{\infty}{\tdens - \edens_n^\param}}
& \leqslant
31L\param^{-\beta/2} +
C_{\nu} M\frac{\log n}{\sqrt{n}}\eqsp .
\end{align*}The proof of Theorem~\ref{Thm:NonAdaptive} is complete.
\subsection{Proof of Theorem~\ref{Thm:Adaptive}}

\subsubsection{Conclusion of the proof} As $B(M)$ and $V(M)$ are non-negative and respectively non-increasing and non-decreasing, Lemma~\ref{lem:GL} applies and one has therefore, on $\Omega$,
\[
\norm{\infty}{\tdens - \edens_n^{\widehat{M}_n}}\leqslant 5 \inf_{M\in\{1,\ldots,n-1\}}\accol{31L\param^{-\beta/2}+\frac{M+1}{N_n^{\param}}\sqrt{3n\log(n+1)}}\eqsp.
\]
Let $M_{\max}\leqslant (n/3 C_\nu\log n)^{1/2\kappa}$. A union bound in \eqref{eq:BoundNnM} gives that, if $X$ is transient,
\[
\Pbf^\nu\parenth{\forall M\in\{1,\ldots,M_{\max}\},\quad \frac{N_n^M}n\geqslant \frac1{C_\nu M^{\kappa}}}\geqslant 1-\sum_{M=1}^{M_{\max}}e^{-n/C_\nu M^{2\kappa}}\geqslant 1-M_{\max}e^{-3 \log n}\enspace.
\]
Therefore, on an event with probability larger than $1-(M_{\max}+1)e^{-3 \log n}-n^{-2}\geqslant 1-2n^{-2}$,
\begin{align*}
\norm{\infty}{\tdens - \edens_n^{\widehat{M}_n}}&\leqslant \inf_{M\in\{1,\ldots,M_{\max}\}}\accol{31L\param^{-\beta/2}+C_\nu M^{\kappa+1}\sqrt{\frac{\log n}{n}}}\\
&\leqslant C_{\nu,\alpha}\parenth{\frac{\log n}n}^{\beta/(4+4\kappa+2\beta)}\eqsp.
\end{align*}
The proof is concluded in the transient case since $\norm{\infty}{\edens_n^{\widehat{M}_n}}\leqslant n+1$.

Since $N_n^{\param}\geqslant N_n^n$ for any $\param\in\{1,\ldots,n-1\}$, in the recurrent case, one has therefore, by \eqref{eq:NnMRec}, on an event with probability larger than $1-C_{\nu}\log n/\sqrt{n}$,
\begin{align*}
\norm{\infty}{\tdens - \edens_n^{\widehat{M}_n}}&\leqslant \inf_{M\in\{1,\ldots,n-1\}}\accol{31L\param^{-\beta/2}+C_\nu M\frac{\log n}{\sqrt{n}}}\\
&\leqslant C_{\nu}\parenth{\frac{(\log n)^2}n}^{\beta/(2\beta+4)}\eqsp.
\end{align*}

\section*{Acknowledgements}
The work of Antoine Havet was supported by a public grant as part of the \textit{Investissement d’avenir} project, reference ANR-11-LABX-0056-LMH, LabEx LMH.\\
The work of Matthieu Lerasle was supported by a public grant from FAST-BIG project ANR-17-CE23-0011.
\bibliographystyle{abbrvnat}
\bibliography{BiblioSource}

\end{document}